\documentclass{amsart}
\usepackage{amsfonts,amsmath,amsthm,color, amssymb}
\usepackage[all]{xy}
\usepackage[pagebackref]{hyperref}

\newtheorem{theorem}{Theorem}[section]
\newtheorem{lemma}[theorem]{Lemma}
\newtheorem{proposition}[theorem]{Proposition}
\newtheorem{corollary}[theorem]{Corollary}

\theoremstyle{definition}

\newtheorem{example}[theorem]{Example}

\theoremstyle{remark}
\newtheorem{remark}[theorem]{Remark}

\numberwithin{equation}{section}

\newcommand{\Oh}{\mathcal{O}}

\newcommand{\sI}{\mathcal{I}}

 \newcommand{\Ext}{\operatorname{Ext}}
\newcommand{\Def}{\operatorname{Def}}

\newcommand{\C}{\mathbb{C}}

\newcommand{\R}{\mathbb{R}}
\newcommand{\Z}{\mathbb{Z}}

\newenvironment{acknowledgement}{\par\addvspace{17pt}\small\rm
\trivlist\item[\hskip\labelsep{\it Acknowledgement.}]}
{\endtrivlist\addvspace{6pt}}

\begin{document}


\title{Deformations of Calabi-Yau manifolds in Fano toric varieties}

\begin{abstract}
In this article, we investigate deformations of a Calabi-Yau manifold $Z$ in a toric variety $F$, possibly not smooth. In particular, we prove that the forgetful morphism from the Hilbert functor $H^F_Z$ of infinitesimal deformations of $Z$ in $F$ to the  functor of infinitesimal deformations of $Z$ is smooth. This implies the smoothness of $H^F_Z $ at the corresponding point in the Hilbert scheme. Moreover, we give some examples and include some computations on the Hodge numbers of Calabi-Yau manifolds in Fano toric varieties.
\end{abstract}

\author{Gilberto Bini}
\address{\newline
Universit\`a degli Studi di Palermo,\hfill\newline
Dipartimento di Matematica e Informatica,\hfill\newline
Via Archirafi n. 34,
90123 Palermo, Italy.}
\email{gilberto.bini@unipa.it }

\author{Donatella Iacono}
\address{\newline  Universit\`a degli Studi di Bari,
\newline Dipartimento di Matematica,
\hfill\newline Via E. Orabona 4,
70125 Bari, Italy.}
 \email{donatella.iacono@uniba.it }

\maketitle

\section{Introduction}

In this paper, we focus our attention on Calabi-Yau manifolds, i.e., projective manifolds with trivial canonical bundle and without holomophic $p$-forms. More precisely, if we focus on dimension greater than or equal to three, $Z$ is a Calabi-Yau manifold of dimension $n$ if the canonical bundle $K_Z := \Omega^n_Z$ is trivial and $H^0(Z, \Omega^p_Z)$ vanishes for $p$ in between $0$ and $n$. Since the canonical bundle is trivial, $Z$ has unobstructed deformations, i.e., the moduli space of deformations of $Z$ is smooth. This is the famous Bogomolov-Tian-Todorov Theorem \cite{bogomolov, tian,todorov}. A more algebraic proof of this fact 
\cite{kawa,zivran,algebraicBTT}  shows that the functor $\Def_Z$ of infinitesimal deformations of  $Z$ is smooth too. In particular, the dimension of the moduli space at the point corresponding to $Z$  is  the dimension of $H^1(Z,T_Z)$, where $T_Z$ denotes the tangent bundle of $Z$. Although we know that the moduli space is smooth, we still miss a geometric understanding of it; for instance, the number of its irreducible components is unknown. A famous conjecture by M. Reid claims that the moduli space of simply connected smooth  Calabi-Yau threefolds is connected via conifold transitions \cite{reid}. The general picture is still unknown but in some cases there has been quite a lot of progress. For example, the moduli spaces of complete intersection Calabi-Yau 3-folds  in products of projective spaces are connected with each other by a sequence of conifold transitions  (see \cite{wang} and references therein).

If $Z$ is contained in an ambient manifold $X$, we can investigate the deformation functor $H^{X}_Z$ of deformations of  $Z$ in  $X$  (fixed) and the forgetful functor $ \phi: H^X_Z \to \Def_Z $, which associates with an infinitesimal deformation of $Z$ in $X$ the isomorphism class of the deformation of $Z$. For example, if $\phi$ is smooth we can conclude that all deformations of $Z$ lie in $X$ and, since $\Def_Z $ is smooth, the functor $H^X_Z$ is also smooth \cite[Proposition 2.2.5]{Sernesi}.

For every Calabi-Yau manifold $Z$ of dimension at least 3 in projective space, the embedded deformations of $Z$ in $\mathbb{P}^n$ are unobstructed. This follows from the vanishing  $ H^1(Z, T_{\mathbb{P}^n|Z})=0$
\cite[Corollary A.2]{huy} that implies that the forgetful morphism $\phi: H^{\mathbb{P}^n}_Z \to \Def_Z $ is smooth, i.e., all deformations of $Z$ as an abstract variety are contained in $\mathbb{P}^n$. Note that dimension at least 3 is fundamental, since in dimension 2 the same statement does not hold  (see also Remark \ref{remark k3}). Moreover, since $ \Def_Z $ is smooth, we can conclude that $H^{\mathbb{P}^n}_Z$ is also smooth. Note that this does not imply that any two Calabi-Yau manifolds of the same dimension in $\mathbb{P}^n$ are deformation equivalent: for instance, explicit examples of threefolds in $\mathbb{P}^6$ that are not deformation equivalent are constructed in \cite{bertin}. 

The projective space $\mathbb{P}^n$ is a toric Fano manifold, i.e., a smooth toric variety with ample anticanonical bundle. Therefore, it is natural to investigate whether the previous results for ${\mathbb{P}^n}$ can be generalised to any toric Fano variety $F$, not simply ${\mathbb{P}^n}$ or the smooth ones.
The interest in toric Fano varieties is motivated both from the mathematics and the physics viewpoint; indeed these varieties have an essential role in the Minimal Model Program and Mirror Symmetry (see, for instance, \cite{rossi} for a recent work on the latter topic). In \cite{petracci}, the author investigates deformation theory of toric Fano varieties.

 In \cite{bi}, we investigated Calabi-Yau manifolds that are anticanonical divisors in toric Fano manifolds of dimension greater than or equal to 4. In particular, we proved that the forgetful morphism 
$ \phi: H^F_Z \to \Def_Z $ is smooth, i.e., all deformations of $Z$ as abstract variety are contained in $F$ \cite[Proposition 1]{bi}.

In this paper, we generalise these results considering as ambient space  a projective simplicial toric Fano variety $F$ and as subvariety a Calabi-Yau  manifold  $Z$  embedded in the Zariski open set of regular points of $F$. Under this assumption  we  investigate  the forgetful morphism 
$ \phi: H^F_Z \to \Def_Z $. In particular,  the following holds (Theorem \ref{teo.forget smooth}).
\begin{theorem} 
Let $F$ be a projective simplicial toric Fano variety   with  $K_{F}= - \sum_{\rho \in \Sigma(1)} D_{\rho}$ its canonical bundle and $Z\subset F$ a Calabi-Yau sumbanifold of dimension greater than or equal to 3, embedded in the Zariski open set of regular points of $F$.  If for all $ \rho \in \Sigma(1)$ we have 
\[
H^1(Z, \Oh_{F} (D_{\rho})\otimes \Oh_Z)=0,
\]
then, the forgetful morphism
$
 \phi: H^{F}_Z \to \Def_Z
$
is smooth.
\end{theorem}

In particular, if $Z$ is a Calabi-Yau manifold, of dimension greater than or equal to 3, which is a complete intersection of very ample divisors,   then  the previous theorem applies if the restriction of all $D_{\rho}$ to $Z$ are nef divisors (Corollary \ref{corollario complete intrsection}). 
We prove Theorem \ref{teo.forget smooth} by showing the vanishing $H^1(Z,T_{F|Z})=0$, which is a sufficient condition for the smoothness of the  forgetful morphism $  \phi: H^{F}_Z \to \Def_Z$. This implication is well known for the smooth case , see for example \cite[Proposition 3.2.9]{Sernesi}. It can be also proved via Horikawa's co-stability theorem for the inclusion morphism $Z\hookrightarrow F$ \cite{hori} or \cite[Section 3.4.5]{Sernesi}.  For the reader's convenience, we give an explicit proof of this fact under our assumptions (see Theorem \ref{firsttheorem}). Note that the vanishing $H^1(Z,T_{F|Z})=0$ is not a necessary condition for the smoothness of the  forgetful morphism $  \phi: H^{F}_Z \to \Def_Z$ (Remark \ref{remark H1=0 non necessario}).

\bigskip
 
In \cite{bertin}, the author also focuses her attention on Calabi-Yau threefolds of codimension 4  in $\mathbb{P}^7$ with Picard number equals to $1$. Using Commutative Algebra methods, new examples are built and their Hodge numbers are investigated. Then, following this approach, we devote our attention to the computation of Hodge numbers of Calabi-Yau submanifolds  $Z$ in a toric Fano variety $F$. In particular, our calculations focus on the cases with $ H^1(Z, T_{F|Z})=0$ and $\dim Z=3,4$ (Section \ref{section conti hodge}). These includes some examples of Calabi-Yau threefold in weighted projective spaces (Section \ref{sezione weighted p spaces}).

Throughout  the paper, we work over the field of complex numbers.
If not otherwise stated, by a toric variety  $F$ we mean a projective simplicial toric Fano variety $F$. We denote by $Z$ a sumbanifold  of $F$ embedded in the Zariski open set of regular points of $F$; thus, $Z$ can be covered by smooth affine open sets.

 In Section \ref{section embedding} we collect some results on toric Fano varieties  and we prove the main theorem on the smoothness of the forgetful functor (Theorem
\ref{firsttheorem}).
Section \ref{section esempi} is devoted to examples of Calabi-Yau submanifolds $Z$ in toric Fano variety $F$, such that  the forgetful functor is smooth.
Finally, Section \ref{section conti hodge} contains some computations on the Hodge numbers of Calabi-Yau threefolds and fourfolds in a toric Fano variety. In particular, we describe examples of Calabi-Yau threefolds in weighted projective spaces and complete intersections fourfolds.

\section{Embeddings in Fano varieties}\label{section embedding}
  In this section, we will follow the notation of the book \cite{coxlittle}, we refer the reader to this book and especially to Chapter 4 for further details. 
Let $F$ be a projective  simplicial toric variety with no torus factors, i.e.,  $\{u_\rho\ | \rho \in \Sigma(1)\}$ spans $N_\R$, where $\Sigma$ is the fan of $F$ in $N_\R$ and $\Sigma(1)$  denotes the 1-dimensional cones of  $\Sigma$.  
 We recall that $F$ is simplicial  when   every $\sigma \in \Sigma$ is simplicial, meaning that the minimal generators of $\sigma$ are linearly independent over $\R$ \cite[pag.180]{coxlittle}.

Moreover, for any strongly convex cone $\sigma \in N_\R$, we denote by $\sigma(1)$ its rays. Also, under our assumptions it makes sense to talk about the canonical divisor $K_{F}$, which can be written as $K_{F}= - \sum_{\rho \in \Sigma(1)} D_{\rho}$ (for further details, we refer the reader to \cite[Chapter 4]{coxlittle}). The variety $F$ is Fano if its anticanonical divisor $-K_F$ is ample. Note that in this case $F$ has no torus factors.
Let ${{\hat{\Omega}}^1}_{F}$ be the sheaf of Zariski 1-differentials. Recall that ${{\hat{\Omega}}^1}_{F}$ is the double dual of the sheaf of K\"ahler differentials ${\Omega}^1_{F}$. Moreover, as proved for instance in \cite[ p. 56]{popa}, the dual of ${{\hat{\Omega}}^1}_{F}$ and the dual of ${\Omega}^1_{F}$ are isomorphic and we denote it by  $T_{F}:=Hom({{\hat{\Omega}}^1}_{F}, \Oh_{F}) =  Hom({{\Omega}^1}_{F}, \Oh_{F}) $.
 
 The hypothesis that $F$ is a  simplicial toric variety with no torus factors is needed for the existence of a generalized Euler exact sequence  \cite[Theorem 8.1.6]{coxlittle}, as in the case for projective spaces;  namely:
\begin{equation}\label{generalized euler for singular}
 0 \to {{\hat{\Omega}}^1}_{F}    \to  \bigoplus_{\rho \in \Sigma(1)}   \Oh_{F}(-D_{\rho}) \to CL({F} )\otimes_{\Z}  \Oh_{F}\ \to 0,
\end{equation}
where $CL(F)$ denotes the divisor class group of $F$.

\begin{theorem}
\label{firsttheorem}
Let $F$ be a simplicial toric Fano variety and  $Z$ a smooth subvariety embedded in the Zariski open set of regular points of $F$. Then, the deformation functor $\phi: H^{F}_Z \to \Def_Z$ is smooth if $H^1(Z, {T_{F} }_{|Z})=0$.
\end{theorem}
\begin{proof}
We follow the usual approach. Consider the  generalized Euler exact sequence  \eqref{generalized euler for singular}:
\[
 0 \to {{\hat{\Omega}}^1}_{F}    \to  \bigoplus_{\rho \in \Sigma(1)}   \Oh_{F}(-D_{\rho}) \to CL({F} )\otimes_{\Z}  \Oh_{F}\ \to 0.
\]
Note that  the divisor class group $CL(F)$ of $F$
 is a finitely generated abelian group that can have torsion \cite[p. 172]{coxlittle}. We denote by $t$ the rank of   $CL(F)$.
 Consider the dual of the above exact  sequence \eqref{generalized euler for singular}, i.e., apply the functor $Hom_{\Oh_{F}}(-,  \Oh_{F})$ to obtain
\begin{equation}\label{dual generalized euler for singular}
 0 \to  \Oh_{F} ^{\oplus t} \to  \bigoplus_{\rho \in \Sigma(1)}   \Oh_{F}(D_{\rho})
\to Hom_{\Oh_{F}}({{\hat{\Omega}}^1}_{F} ,  \Oh_{F}) \to Ext^1_{\Oh_{F}}( CL({F} )\otimes_{\Z}  \Oh_{F} , \Oh_{F})\cdots .
\end{equation}
Note that the sheaf $ Ext^1_{\Oh_{F}}( CL({F} )\otimes_{\Z}  \Oh_{F} , \Oh_{F})=0$. Since $CL(F)$ is a finitely generated group, the torsion subgroup is a finite abelian group, so it is a finite sum of cyclic groups of prime power order $r=p^h$. Thus, tensoring by $\Oh_{F}$ over $\Z$, we have
 \[ 0 \to r\Z\otimes _\Z \Oh_F \to \Oh_{F} \to  \Z_r \otimes_\Z \Oh_{F} \to 0. \]
Since $\Z_r \otimes_\Z \Oh_{F}$ is a torsion sheaf, applying $Hom_{\Oh_{F}}(-,  \Oh_{F})$,  we get
\[ 0 \to Hom_{\Oh_{F}}(\Oh_{F},  \Oh_{F})  \to Hom_{\Oh_{F}}( r\Z\otimes _\Z \Oh_F,  \Oh_{F})  \to Ext^1_{\Oh_{F}}( \Z_r \otimes_Z \Oh_{F} , \Oh_{F}) \to 0. \]
Moreover, the map on the LHS is an isomorphism of sheaves, we conclude that  the sheaf $Ext^1_{\Oh_{F}}( \Z_r \otimes_Z \Oh_{F} , \Oh_{F})=0$.

 In addition,  as mentioned before, $T_{F}=Hom({{\hat{\Omega}}^1}_{F}, \Oh_{F}) = 
 Hom({{\Omega}^1}_{F}, \Oh_{F})$, i.e., the tangent sheaf of the Fano variety. Then, the exact sequence  \eqref{dual generalized euler for singular} reduces to
\begin{equation}\label{risoluzione tangente fano}
  0 \to  \Oh_{F} ^{\oplus t} \to  \bigoplus_{\rho \in \Sigma(1)}   \Oh_{F}(D_{\rho})
\to T_{F} \to 0.
\end{equation}
Let $Z$ be a smooth variety contained in the smooth locus of the variety $F$ such that the inclusion $j: Z \hookrightarrow F$ is a closed embedding with ideal sheaf $\sI \subset \Oh_{F}$. Since $Z$ is smooth and embedded in the Zariski open set of regular points of $F$, the ideal $\sI/\sI^2$ is locally free  by \cite[Exercise 17.12]{eisenbud}. Then, under these assumptions, the conormal sequence is exact \cite[Theorem C.15. (iii)]{Sernesi}, namely:
 \begin{equation}\label{conormal exact}
0\to \sI/\sI^2 \to j^*\Omega^1_{F} \to \Omega^1_Z \to 0.
\end{equation}
Now, consider the dual of the exact sequence \eqref{conormal exact},
 \[
 0\to T_Z \to Hom_{  \Oh_{Z}}(j^*\Omega^1_{F},  \Oh_{Z}) \to  Hom_{  \Oh_{Z}}(  \sI/\sI^2,  \Oh_{Z})\to Ext^1_{  \Oh_{Z}}(\Omega^1_Z ,  \Oh_{Z}) \to \cdots.\]
 Since $Z$ is smooth, the sheaf $\Omega^1_Z $ is locally free and so $ Ext^1_{  \Oh_{Z}}(\Omega^1_Z ,  \Oh_{Z})=0$. Therefore, we have the usual normal exact sequence
\begin{equation}\label{successione esatta corta normale tangente}
 0 \to T_Z \to T_{F|Z} \to N_{Z/{F}} \to 0.
\end{equation}
The induced exact sequence in cohomology is given by
\[
\cdots \to  H^0(Z, N_{F/Z} ) \to 
 H^1(Z, T_Z) \to H^1(Z, {T_{F} }_{|Z})\to H^1(Z, N_{F/Z} )  \to   H^2(Z, T_Z).
\]
If $ H^1(Z, {T_{F} }_{|Z})=0$, then the morphism $ H^0(Z, N_{F/Z} ) \to 
 H^1(Z, T_Z)$ is surjective and $H^1(Z, N_{F/Z} )  \to   H^2(Z, T_Z)$ is injective.
 
  According to \cite[Poposition 3.2.1]{Sernesi},  since  $Z\subset X$ is a closed embedding, the Zariski tangent space of $H^Y_Z$ at the point corresponding to $Z$ is $H^0(Z, N_{F/Z} )$. Moreover, since $Z$ is smooth the closed embedding is regular; hence  $H^1(Z, N_{F/Z} )$ is an obstruction space for the Hilbert functor $H^Y_Z$. As a consequence, if $ H^1(Z, {T_{F} }_{|Z})=0$, the forgetful morphism $ \phi: H^{F}_Z \to \Def_Z$ is smooth. In fact, under our assumptions $Z$ is contained in the Zariski open set of regular points of the Fano variety $F$. Thus $Z$ can be covered by smooth affine open sets. This remark allows us to apply the standard  smoothness criterion in deformation theory: see, for instance, \cite[Proposition 3.2.9]{Sernesi} or \cite[Theorem 4.11]{ManettiSeattle} and prove that $\phi$ is smooth.
\end{proof}

\begin{remark}  
If $Z$  is a Calabi-Yau submanifold of $F$, then $\Def_Z$ is smooth (Bogomolov-Tian-Todorov Theorem) of dimension $ H^1(Z, T_Z)$.  Then, by the previous theorem, $H^F_Z$ is also a smooth functor: the deformation space of $Z$ inside $F$ is smooth of dimension $H^0(Z, N_{Z/F})$.

\end{remark}

\begin{remark}\label{remark k3}
If $\dim Z =2$, then Theorem \ref{teo.forget smooth} does not hold.
It is enough to consider a $K3$ surface in $\mathbb{P}^4$. In this case, it is not true that the morphism $\phi: H^F_Z \to \Def_Z$ is smooth \cite[Examples 3.2.11]{Sernesi}, indeed $ H^1(Z, {T_{\mathbb{P}^4} }_{|Z})\neq 0$.
\end{remark}

\begin{remark}\label{remark H1=0 non necessario}
The condition $ H^1(X, T_{F|Z})=0$ is not a necessary condition for the smoothness of the forgetful morphism
$\phi: H^{F}_Z \to \Def_Z$. For example, \cite[Example 3.4.4 (iii)]{Sernesi}, let $Z\cong \mathbb{P}^1\subset F$  be a nonsingular projective curve negatively embedded in a projective nonsingular Hirzebruch surface $F$ with $Z^2= -n<0$, $n\geq 1$. Then, the exact sequence
\[
 0 \to T_Z \to T_{F|Z} \to N_{Z/{F}} \to 0
\]
splits since $\Ext^1_{\Oh_Z}(N_{Z/F}),T_Z )=H^1(Z, \Oh_Z(n+2))=0$ and so $T_{F|Z} \cong \Oh_Z(2)\oplus \Oh_Z(-n)$. This implies that $h^0(Z,T_{F|Z})=3$. Moreover $h^0(Z,N_{Z|F})=0$ and so $Z$ is rigid in $F$ in addition to being rigid as an abstract variety. Then, the morphism induced by $\phi$ on the tangent space is surjective, and it is injective on the obstruction spaces: they are both zero since there are no deformations. In conclusion,  $\phi$ is smooth even if $h^1(Z,T_{F|Z})=h^1(Z,\Oh_Z(2)\oplus \Oh_Z(-n))=n-2 $ can be non-zero.
\end{remark}

\section{A large class of examples}\label{section esempi}

Let $F$ be a simplicial  toric Fano variety of dimension $\dim F=n+m$ for $m \geq 3$. As in the previous section, the dual of the generalised Euler exact sequence  \eqref{risoluzione tangente fano}  for $F$ is
\[
  0 \to  \Oh_{F} ^{\oplus t} \to  \bigoplus_{\rho \in \Sigma(1)}   \Oh_{F}(D_{\rho})
\to T_{F} \to 0,
\]
where $K_F= - \sum_{\rho \in \Sigma(1)} D_{\rho}$ \cite[Theorem 8.2.3]{coxlittle}.

\begin{lemma}\label{lemma vanishing}
Let $F$ be a toric Fano variety and $Z\subset F$ a Calabi-Yau sumbanifold. Let  $D $ be a divisor 
such that $D_{|Z}$ is nef and big, then
\[
H^j(Z, \Oh_F (D) \otimes \Oh_Z )=0, \quad \forall \  j>0.
\]

\end{lemma}

\begin{proof}
Since the divisor $D_{|Z}$ is nef and big and  $Z$ is a Calabi-Yau manifold (and so $K_Z=0$), the
Kawamata-Viehweg vanishing Theorem \cite[Theorem 7.21]{debarre} or \cite[Theorem 9.3.10]{coxlittle}, implies that 
\[
H^j(Z, \Oh_F (D) \otimes \Oh_Z )=H^j(Z, \Oh_Z({D}_{| Z}+K_Z) ) =0, \quad \forall \ j>0.
\]
\end{proof}

\begin{remark}\label{remark nef e big}

Let $F$ be a toric Fano variety and $Z\subset F$ a Calabi-Yau sumbanifold, such that $\dim Z=m$.
If the divisor $D$ is such  $D_{|Z}$ is nef and $
D^m \cdot Z >0$, then ${D}_{| Z}$ is nef and big  \cite[Section 1.29]{debarre} and so we can apply the previous Lemma \ref{lemma vanishing}. Note also that if  $D$ is nef then  its restriction ${D}_{| Z}$  to $Z$ is also nef  \cite[Section 1.6]{debarre}.  

\end{remark}

\begin{remark}
A useful condition for nefness of a divisor $D$ in $F$ is the following:
given a cone $\sigma \in \Sigma$, any nef divisor is linearly equivalent to a divisor of the form 
\[
D= \sum_\rho a_\rho D_\rho, 
\]
where  $a_\rho=0$ if  $\rho\in \sigma(1)$    and  
 $a_\rho\geq0$  for $\rho\not\in \sigma(1)$, see \cite[Equation 6.4.10]{coxlittle}. 
\end{remark}
 
\begin{corollary}\label{coroll vanishinf complete interse}
 Let $F$ be a toric Fano variety of dimension $\dim F=n+m$ and denote by $Z\subset F$ a  Calabi-Yau  smooth variety of dimension $\dim Z=m$. Suppose that $Z$ is a complete intersection of very ample divisors. Then, for any divisor $D$ such that ${D}_{| Z}$ is nef, we have
\[
H^i(Z, \Oh_F (D) \otimes \Oh_Z )=0, \quad \forall \ i >0.
\]
\end{corollary}

\begin{proof}  Suppose that $Z$ is a complete intersection of very ample divisors. Then, there exist $n$ very ample divisors $N_j$, for $j=1, \ldots, n$, such that $Z=Y_1 \cdots Y_n$, where  $Y_j$ is an  element in the linear system $|N_j|$.

In particular, we have 
 \[
D^m \cdot Z = D^m \cdot Y_1 \cdots Y_n = N_1 \cdot(D^m \cdot N_2 \cdots N_n) >0,
\]
where the last equality follows from the Nakai-Moishezon Theorem \cite[Theorem 1.21]{debarre}, indeed $N_1$  is ample and 
$(D^m \cdot N_2 \cdots N_n)$ has dimension 1. Then, the conclusion follows by Remark \ref{remark nef e big} and Lemma \ref{lemma vanishing}.   

\end{proof}

\begin{theorem} \label{teo.forget smooth}
Let  $F$ be a simplicial  toric Fano variety and $Z\subset F$ a Calabi-Yau submanifold of $\dim Z=m$, with $m\geq 3$.
Suppose that for all $ \rho \in \Sigma(1)$ the  divisor $D_{\rho}$ satisfies the following vanishing
\[
H^1(Z, \Oh_F (D_{\rho})\otimes \Oh_Z)=0.
\]
Then, the forgetful morphism  $\phi: H^F_Z \to \Def_Z$ is smooth.
\end{theorem}

\begin{proof}

By tensoring with $ \Oh_Z$ the  dual of the generalized Euler exact sequence for $F$  \eqref{risoluzione tangente fano}, we obtain 
\begin{equation}\label{sequen euler restricted}
  0 \to  \Oh_{Z} ^{\oplus t} \to  \bigoplus_{\rho \in \Sigma(1)}   \Oh_{F}(D_{\rho}) \otimes \Oh_Z  \to 
T_{F|Z} \to 0 ,
 \end{equation}
and so
\[
\cdots  \to
 H^1(Z,   \Oh_{Z} ^{\oplus t} ) \to
 \bigoplus_{\rho \in \Sigma(1)}  H^1(Z, \Oh_F (D_{\rho})  \otimes \Oh_Z ) \to H^1 (Z, T_{F|Z} )
\to \cdots
\]
\[
\cdots  \to
 H^2(Z,  \Oh_{Z} ^{\oplus t})  \to
 \bigoplus_{\rho \in \Sigma(1)}  H^2(X, \Oh_F (D_{\rho})  \otimes \Oh_Z ) \to H^2 (Z, T_{F|Z} )
\to \cdots .
\]
By hypothesis 
$H^1(Z, \Oh_F (D_{\rho}) \otimes \Oh_Z )=0$.
Since $Z$ is a Calabi-Yau manifold of dimension $\dim Z=m\geq 3$, $ H^2(Z, \Oh_Z)=0$ and so   
$ H^2(Z,  \Oh_{Z} ^{\oplus t})=0$; this concludes the proof because $H^1(Z,T_{F|Z})=0$ and we can apply Theorem \ref{firsttheorem}.

\end{proof}


\begin{example} \label{example complete intersection}
 Let $F$ be a toric Fano variety of dimension $\dim F=n+m$ and denote by $Z\subset F$ a  Calabi-Yau submanifold of dimension $\dim Z=m\geq 3$. Suppose that $Z$ is a complete intersection of very ample divisors, such that ${D_{\rho}}_{| Z}$ is nef for all $ \rho \in \Sigma(1)$. Then, by Corollary \ref{coroll vanishinf complete interse} we have that $H^i(Z, \Oh_F (D_{\rho}) \otimes \Oh_Z )=0$ for all $i >0$.

\end{example}

\begin{corollary}
Let $F$ be a simplicial  toric Fano variety   and $Z\subset F$ a Calabi-Yau sumbanifold of $\dim Z=m$. 
 Let  $D_{\rho}$ be the  divisor 
associated with $\rho \in \Sigma(1)$ and assume further that ${D_{\rho}}_{| Z}$ is nef and $D_{\rho}^m \cdot Z >0$, for all $ \rho \in \Sigma(1)$. Then, 
\[
 \phi: H^F_Z \to \Def_Z
\]
is smooth.
 \end{corollary}
 
\begin{proof}
It is enough to apply Lemma \ref{lemma vanishing}
and Theorem \ref{teo.forget smooth}. 
\end{proof}
\begin{corollary}\label{corollario complete intrsection}
 Let  $F$ be a simplicial  toric Fano variety of dimension $\dim F=n+m$ and denote by $Z\subset F$ a  Calabi-Yau submanifold of dimension $\dim Z=m\geq 3$. Suppose that $Z$ is a complete intersection of very ample divisors, such that ${D_{\rho}}_{| Z}$ is nef for all $ \rho \in \Sigma(1)$. Then $\phi: H^F_Z \to \Def_Z$ is smooth.
 \end{corollary}
\begin{proof}
It is enough to apply Theorem \ref{teo.forget smooth} and Example \ref{example complete intersection}.

\end{proof}

\section{Hodge numbers of Calabi-Yau varieties}\label{section conti hodge}

In this section, we are interested in computing Hodge numbers of Calabi-Yau submanifolds $Z$ of a toric Fano variety $F$, in particular for the case investigated in the previous section, i.e., whenever $ H^1(Z, T_{F|Z})=0$.
Recall that the Hodge numbers of $Z$ are defined as  $h^{i,j}(Z)= \operatorname{dim}_{\C} H^j(Z, \Omega^i_Z)$ and they satisfy the Hodge duality $h^{i,j}(Z)=h^{j,i}(Z)$. 
If $\dim Z=m$, since $K_Z=0$, we have  $T_Z \cong \Omega_Z^{m-1}$ and so 
\[h^{m-1,i}(Z)= \dim  H^i(Z, \Omega_Z^{m-1})=   \dim  H^i(Z, T_Z).\]

Under the assumption that  $Z$ is  a Calabi-Yau submanifolds of
 a simplicial  toric Fano variety $F$ such that $ H^1(Z, T_{F|Z})=0$, we can estimate the Hodge numbers of $Z$.

\begin{proposition}\label{prop su H 1 ( TX)}

Let  $F$ be a simplicial  toric Fano variety of dimension $\dim F=n+m$ and denote by $Z\subset F$ a  Calabi-Yau submanifold of dimension $\dim Z=m$.
If  $ H^1(Z, T_{F|Z})=0$, then
\[
h^{m-1,1}(Z)= \dim H^1(Z, T_Z)=\dim  H^0(Z, N_{Z/F})- \bigoplus_{\rho \in \Sigma(1)} \dim  H^0(Z, \Oh_F (D_{\rho}) \otimes \Oh_Z )  +t,
 \]
where $t$ is the  rank of $ CL(F)$.
\end{proposition}

\begin{proof}
The exact sequence \eqref {successione esatta corta normale tangente}
\[
 0 \to T_Z \to T_{F|Z} \to N_{Z/{F}} \to 0
\]
 induces the following exact sequence in cohomology:
\[
0 \to  H^0(Z, T_Z) \to H^0(Z, T_{F|Z}) \to H^0(Z, N_{Z/F})  \to  
 H^1(Z, T_Z) \to H^1(Z, T_{F|Z})=0.
\]
Since  $\dim  H^i(Z, T_Z) = \dim  H^i(Z, \Omega_Z^{m-1})$, we have $ H^0(Z, T_Z)=0$ and this implies that
\[
 \dim H^1(Z, T_Z)=\dim  H^0(Z, N_{Z/F})-  \dim H^0(Z, T_{F|Z}). 
 \]
Then, by the long exact sequence associated with the Euler exact sequence  \ref{sequen euler restricted} restricted to $Z$, we obtain   
\[
0  \to
 H^0(Z, \Oh_{Z} ^{\oplus t})  \to
 \bigoplus_{\rho \in \Sigma(1)}  H^0(Z, \Oh_F (D_{\rho})  \otimes \Oh_Z ) \to H^0 (Z, T_{F|Z} )\to
 \]
 \[
\to  H^1(Z, \Oh_{Z} ^{\oplus t})\to \cdots . 
\]
Since $ H^0(Z,   \Oh_Z) =\C$ and   $H^1(Z,   \Oh_Z)=0$, it follows that
\[
\dim H^0(Z, T_{F|Z})=    \bigoplus_{\rho \in \Sigma(1)} \dim  H^0(Z, \Oh_F (D_{\rho}) \otimes \Oh_Z ) -t,
\]
where $t$ is the  rank of $ CL(F)$. Hence,
\[
 \dim H^1(Z, T_Z)=\dim  H^0(Z, N_{Z/F})- \bigoplus_{\rho_i \in \Sigma(1)} \dim  H^0(Z, \Oh_F (D_{\rho})  \otimes \Oh_Z ) +t. 
 \]

\end{proof}

\begin{remark}
In the setup above of a  smooth Calabi-Yau submanifold $Z$ in  a simplicial  toric Fano variety  $F$, the previous proposition provides the dimension of the moduli space at the point corresponding to $Z$, that is smooth of dimension $H^1(Z, T_Z) $.  
\end{remark}

\begin{proposition}\label{propo h 1,1 }
  Let   $F$ be a simplicial  toric Fano variety  of dimension $\dim F=n+m$ and denote
by $Z\subset F$ a  Calabi-Yau submanifold of dimension $\dim Z=m \geq 3$, that is the complete intersection of $n$ very ample divisors, such that ${D_{\rho}}_{| Z}$ is nef for all $ \rho \in \Sigma(1)$. Then, 
\[
 h^{1,1}= \dim H^{1 }(Z, \Omega^1_Z)= t,
\]
where $t$  is the rank of $CL(F)$.
\end{proposition}

\begin{proof}
As above, since $Z$ is a Calabi-Yau manifold of dimension $m$, we have that $\Omega^m_Z\cong \Oh_Z$ 
and so $T_Z \cong \Omega^{m-1}_Z$. Therefore, 
 \[
 h^{1,1}= \dim H^{1 }(Z, \Omega^1_Z)= \dim H^{m-1 }(Z, \Omega^{m-1}_Z)=\dim H^{m-1 }(Z, T_Z).
\]
The submanifold $Z$ is the complete intersection of $n$ very ample divisors $N_1, \cdots ,N_n$, i.e., $Z=  N_1 \cdot \cdots \cdot N_n$.
In particular, 
$ N_{Z/F}= \Oh_Z(N_1) \oplus \cdots \oplus  \Oh_Z(N_n),$ and so
\[
 H^{j}(Z, N_{Z/F})= \bigoplus_{i=1}^n  H^{j}(Z,  \Oh_Z(N_i))= \bigoplus_{i=1}^n  H^{j}(Z,  \Oh_Z(N_i+K_Z))=0 \quad \forall j >0.
\]
where in the last equality we use the Kodaira vanishing  (the restriction of an ample line bundle to a closed subscheme is still ample).
The long exact sequence in cohomology associated with \eqref{successione esatta corta normale tangente}
implies that  
 \[
\cdots  \to H^{m-2}(Z, N_{Z/F})  \to 
 H^{m-1}(Z, T_Z) \to H^{m-1}(Z, T_{F|Z}) \to H^{m-1}(Z, N_{Z/F})  \to \cdots   \]
and so $ H^{m-1}(Z, T_Z) \cong  H^{m-1}(Z, T_{F|Z})$.
 (We actually have 
$ H^{j}(Z, T_Z)= H^j  (Z, T_{F|Z} )$, for all  $j=m-1>2$).

Finally, by the long exact sequence associated with the Euler exact sequence  \ref{sequen euler restricted} restricted to $Z$, we obtain   
\[
\cdots  \to \!\!\!\!
 \bigoplus_{\rho \in \Sigma(1)} \!\!\! H^{m-1}(Z, \Oh_F (D_{\rho})  \otimes \Oh_Z )
 \to\! H^{m-1} (Z, T_{F|Z} ) \!\to \!
 H^{m}(Z,  \Oh_{Z} ^{\oplus t}) \!  \to\!\!\!
 \bigoplus_{\rho \in \Sigma(1)} \!\!\! H^{m }(Z, \Oh_F (D_{\rho})  \otimes \Oh_Z )
\to \cdots.
\]
Corollary \ref{coroll vanishinf complete interse} implies $ H^{j }(Z, \Oh_F (D_{\rho})  \otimes \Oh_Z )=0$, for all $j>0$ and all  $\rho \in \Sigma(1)$;
therefore 
 \[
 h^{1,1}=  \dim H^{m-1 }(Z, T_Z) = \dim H^{m-1 }(Z, T_{F|Z} )= \dim  H^{m}(Z, \Oh_{Z} ^{\oplus t}) = t,
\] 
since  $ H^{m}(Z, \Oh_Z)=\C$.
 
\end{proof}

 \begin{remark}\label{remark numeri hodge 3 fold complete intersection}
If $Z$ is a Calabi-Yau submanifold of dimension $\dim Z=3$, that is a complete intersection of $n$ very ample divisors in  
 a simplicial  toric Fano variety $F$   of dimension $\dim F=n+3$, such that ${D_{\rho}}_{| Z}$ is nef for all $ \rho \in \Sigma(1)$, then we can describe the Hodge diamond of $X$.
Indeed, by  the previous proposition we computed $h^{1,1}$ and
 by Proposition \ref{prop su H 1 ( TX)} we can compute  $h^{1,2}(Z)=h^{2,1}(Z)=\dim H^1 (Z,\Omega_Z ^2) = \dim H^1 (Z,T_Z)$.\end{remark}

  \begin{remark}
The weighted  projective spaces are examples of projective toric varieties  with Picard number 1.
Therefore, under the assumption of the previous proposition, the complete intersection Calabi-Yau manifolds  in weighted  projective spaces have $h^{1,1}=1$. If we require $F$ smooth, then $F$ is the projective space $\mathbb{P}^n$, that are the only smooth projective toric variety  with Picard number 1 \cite[Exercise 7.3.10]{coxlittle}.
\end{remark}
 
\subsection{Examples of Calabi-Yau threefolds in weighted projective spaces}\label{sezione weighted p spaces}

Let $P=\mathbb{P}(1,1,1,a_3,\ldots , a_n)$ be the  weighted projective spaces for $n\geq 3$, and $a_i \geq 0$, for all $i\geq 3$.
According to \cite[Claim 37 ] {kollar} or  \cite[Example 3.1.25]{Sernesi},   we have   $T^1_P=0$ and $H^1(P,T_P)=0$ and so the local deformations of  $P$ are trivial.

In \cite{fletcher}, there are examples of smooth varieties with trivial canonical bundle in various weighted projective space.  In particular,
there are  the following nonsingular threefolds weighted hypersurfaces  \cite[Theorem 14.3]{fletcher}:
 \[
 X_5 \subset \mathbb{P}(1,1,1,1,1) \quad 
  X_6 \subset \mathbb{P}(1,1,1,1,2) \quad  X_8 \subset \mathbb{P}(1,1,1,1,4) \quad  X_{10} \subset \mathbb{P}(1,1,1,2,5),
 \]
and  the following  nonsingular codimension 2 weighted threefolds complete intersection  \cite[Theorem 14.6]{fletcher}:
 \[
 X_{2,4} \subset \mathbb{P}(1,1,1,1,1,1) \qquad  X_{3,3} \subset \mathbb{P}(1,1,1,1,1,1) \]
  \[ 
 X_{3,4} \subset \mathbb{P}(1,1,1,1,1,2) \qquad   X_{4,4} \subset \mathbb{P}(1,1,1,1,2,2).
 \]
These varieties are all examples  of smooth subvarieties with trivial canonical bundle in a toric Fano variety with Picard rank one  (that is not smooth except the cases of projective spaces $ \mathbb{P}(1,1,1,1,1)$  and $ \mathbb{P}(1,1,1,1,1,1) $). 
Moreover, if $Z$ is any of these Calabi-Yau threefolds  we have $H^1(Z, {T_P}_{|Z})=0$. Indeed, the generalised Euler exact sequence \eqref{risoluzione tangente fano}  for $P=\mathbb{P}(1,1,1,a_3,\ldots , a_n)$ is 
\[
  0 \to  \Oh_{P} \to  \bigoplus_{i\geq 3}   \Oh_{P}(a_i) \oplus  \Oh_{P}^{\oplus 3}
\to T_{P} \to 0,
\]
and it restrict to 
\begin{equation}\label{euler weighted projective space on Z}
  0 \to  \Oh_{Z} \to  \bigoplus_{i\geq 3}   \Oh_{Z}(a_i) \oplus  \Oh_{Z}^{\oplus 3}
\to  {T_P}_{|Z} \to 0.
\end{equation}
Considering the long exact sequence associated with \eqref{euler weighted projective space on Z},
it is enough to prove 
\begin{equation}\label{vanishing}
 H^1( Z,\bigoplus_{i\geq 3}   \Oh_{Z}(a_i) \oplus  \Oh_{Z}^{\oplus 3})= H^2(Z, \Oh_Z)=0.
\end{equation}
For the weighted hypersurface case $Z=X_d$, we tensorize the exact sequence
\[
  0 \to  \Oh_{P}(-d) \to    \Oh_{P}
\to  \Oh_{Z}\to 0,
\]
with $ \Oh_{P}(a)$ and we conclude the vanishing \eqref{vanishing}, since
 $H^i(P, \Oh_{P}(n))=0$ for all $n \in \Z$ and  $i \neq 0, \sum_i a_i +3$  \cite[Section 1.4]{dolga}.
For $Z$ any of the above codimension 2 weighted threefolds complete intersection, similar computations prove the vanishing $H^1(Z, {T_P}_{|Z})=0$.
Therefore, as proved in Theorem \ref{firsttheorem} the forgetful functor $
 \phi: H^P_Z \to \Def_Z $ is smooth and so the functor $ H^P_Z $ is smooth at the corresponding point.

\begin{remark}
A we  also noted in Remark \ref {remark numeri hodge 3 fold complete intersection}, we can compute the Hodge numbers of these Calabi-Yau weighted complete intersections.
By Proposition \ref{propo h 1,1 }, we have 
\[ h^{1,1}= \dim H^{1 }(Z, \Omega^1_Z)= 1,\]
and  by Proposition \ref{prop su H 1 ( TX)}, we have 
\[h^{1,2}(Z)=h^{2,1}(Z)=\dim  H^0(Z, N_{Z/P})- \bigoplus_{\rho \in \Sigma(1)} \dim  H^0(Z, \Oh_Z (D_{\rho})) +1.\]
\end{remark}

\subsection{Examples of complete intersection Calabi-Yau fourfolds}

Let $Z$ be a smooth Calabi-Yau fourfold in $F$; let $j:Z \rightarrow F$ be a closed embedding of $Z$ in $F$. Suppose further that $Z$ is a complete intersection of $n$ very ample divisors $N_1, \ldots, N_n$ so that $\dim(F)=n+4$. Let us analyse $h^{i,j}(Z)= \operatorname{dim}_{\C} H^j(Z, \Omega^i_Z)$. By Proposition \ref{propo h 1,1 },  we have   $h^{1,1}(Z)= t=rank CL(F)$. Next, let us compute  $ h^{1,2}(Z)= h^{2,1}(Z)$.

\begin{proposition} \label{Prop h 1 , 2  }
  Let  $F$ be a simplicial  toric Fano variety of dimension $\dim F=n+m$ and denote
by $Z\subset F$ a  Calabi-Yau submanifold of dimension $\dim Z=m\geq 4$, that is the complete intersection of $n$ very ample divisors, such that ${D_{\rho}}_{| Z}$ is nef for all $ \rho \in \Sigma(1)$. Then, 
\[
 h^{1,2}= h^{2,1}=0.
\]

\end{proposition}

\begin{proof}

The proof goes as in Proposition \ref{propo h 1,1 }.
Since $Z$ is a Calabi-Yau manifold of dimension $m$, we have that $\Omega^m_Z \cong \Oh_Z$ 
and so $T_Z \cong \Omega^{m-1}_Z$. Therefore, 
 \[
 h^{1,2}= \dim H^{2 }(Z, \Omega^1_Z)= \dim H^{m-2 }(Z, \Omega^{m-1}_Z)=\dim H^{m-2 }(Z, T_Z).
\]
As in the proof of Propostion \ref{propo h 1,1 }, $H^{j}(X, N_{Z/F})=0 $ for all $j >0$.
The long exact sequence in cohomology associated with \eqref{successione esatta corta normale tangente}
implies that
 \[
\cdots  \to H^{m-3}(Z, N_{Z/F})  \to 
 H^{m-2}(Z, T_Z) \to H^{m-2}(Z, T_{F|Z}) \to H^{m-2}(Z, N_{Z/F})  \to \cdots   \]
and so, since $m\geq 4$, $ H^{m-2}(X, T_Z) \cong  H^{m-2}(Z, T_{F|Z})$.

Finally,  tensoring with $ \Oh_Z$ the generalized Euler exact sequence for $F$ \eqref{risoluzione tangente fano}, we obtain
\[
\cdots  \to
 \bigoplus_{\rho_i \in \Sigma(1)}  H^{m-2}(Z, \Oh_F (D_{\rho})  \otimes \Oh_Z )
 \]
 \[\to H^{m-2} (Z, T_{F|Z} ) \to 
 H^{m-1}(Z,  \Oh_{Z} ^{\oplus t})  \to
 \]
 \[
 \bigoplus_{\rho_i \in \Sigma(1)}  H^{m-1 }(Z, \Oh_F (D_{\rho})  \otimes \Oh_Z )
\to \cdots,
\]
where $t$ is the rank of $CL(F)$.
Corollary \ref{coroll vanishinf complete interse} implies $ H^{j }(Z, \Oh_F (D_{\rho})  \otimes \Oh_Z )=0$, for all $j>0$ and all  $\rho \in \Sigma(1)$;
therefore 

 \[ h^{1,2}= \dim H^{m-2 }(Z, T_Z) = \dim H^{m-2 }(Z, T_{F|Z} ) =  \dim  H^{m-1}(Z,  \Oh_{Z} ^{\oplus t}) =0.\]
\end{proof}

Since $Z$ is a Calabi-Yau manifold of dimension $4$, we have that $\Omega^4_Z \cong \Oh_Z$ 
and so $T_Z \cong \Omega^{4-1}_Z\cong \Omega^{3}_Z$. 
Therefore,  by Proposition \ref{prop su H 1 ( TX)}, we have 

\[\begin{split}
 h^{1,3}(Z)&= h^{3,1}(Z)= \dim H^{1 }(Z, \Omega^3_Z)= \dim H^{1 }(Z,T_Z)\\
&= \dim  H^0(Z, N_{Z/F})- \bigoplus_{\rho \in \Sigma(1)} \dim  H^0(Z, \Oh_F (D_{\rho}) +t .
\, \end{split}
\]
If  we denote by $c= h^{1,3}(Z)$ and $d = h^{2,2}(Z)$,   the Hodge diamond of Z is
\[
 \xymatrix@R=8pt@C=6pt{
 &  &   &   & 1   \\
  & &   & 0 &   & 0 \\
  & & 0 &   &  t &   & 0 \\
&0  &   &  0 &   & 0 &   & 0 \\
1  &   &  c & &d&   & c&   & 1 \\
&0 &   &  0  &   & 0 &   &0 \\
 &  & 0 &   &  t &   & 0 \\
  & &   & 0 &   & 0 \\
   & &   &   & 1.
}
\] 
Then, we only miss the computation of $d$.

\begin{proposition}
Assume $Z$ is a  Calabi-Yau manifold of dimension 4, then we have
\begin{equation}\label{eq d classi chern}
d= h^{2,2}(Z)=2c-2 + \frac{1}{45}\left(3c_2^2(Z)+14c_4(Z) \right).
\end{equation}
\end{proposition}

\begin{proof}
 In order to compute $d=h^{2,2}(Z)$, we use the signature $\sigma(Z)$ of the (complex) four manifold $Z$. On the one hand, $\sigma(Z)$ is defined as
$$
\sigma(Z)=h^{0,4}(Z)-h^{1,3}(Z)+h^{2,2}(Z) -h^{3,1}(Z) + h^{4,0}(Z),
$$
which reads as $\sigma(Z)=2-2c+d$. On the other hand, the Hirzebruch Signature Theorem (see, for instance, \cite{hst}) gives
$$
\sigma(Z)= \frac{1}{45}\left(7p_2(Z)-p_1^2(Z)  \right),
$$
where the $p_j(Z)$'s are the Pontryagin numbers. These are related to the Chern classes of $Z$ as follows:
$$
p_1(Z)=-2c_2(Z), \qquad p_2(Z)= 2c_4(Z) + c_2^2(Z).
$$
Hence we obtain the statement.

\end{proof}

\begin{remark}
Assume $Z$ is a smooth Calabi-Yau fourfold in $F$; let $j:Z \rightarrow F$ be closed embedding of $Z$ in $F$. Suppose further that $Z$ is the complete intersection of $n$ very ample divisors $N_1, \ldots, N_n$ so that $\dim(F)=n+4$, then  the Chern classes of $Z$  can be computed in terms of the Chern classes of the bundles $N_j$. More precisely, the following recursive relations can be deduced from the exact sequence defining the normal bundle, namely:
$$
c_2(Z)=j^*c_2(F)-c_2({\mathcal N}_{F/Z}),
$$
$$
c_3(Z)=j^*c_3(F)-c_3({\mathcal N}_{F/Z})-c_2(X)c_1({\mathcal N}_{F/Z}),
$$
$$
c_4(Z)=j^*c_4(F)-c_2(Z)c_2({\mathcal N}_{F/Z})-c_3(Z)c_1({\mathcal N}_{F/Z})-c_4({\mathcal N}_{F/Z}).
$$

In \cite{green,green2,ghl},  the authors carry out computations on Chern classes and Hodge numbers for Calabi-Yau fourfold that are complete intersection in product of projective spaces.

\end{remark}

\begin{acknowledgement}
 The authors wish to thank Marco Manetti for useful discussions during the preparation of this paper and Tristan H\"ubsch for drawing their attention to \cite{wang,berglund, green2, rossi}. They are also very grateful to Andrea Petracci  for useful comments and for pointing out an inaccuracy in the first version of this paper. They aknowledge the referee for suggestions and comments that improved the exposition of the paper.

 \end{acknowledgement}


\begin{thebibliography}{FMM10}





 \bibitem[BH18]{berglund}  
P. Berglund, T. H\"ubsch,  \emph{On Calabi-Yau generalized complete intersections from Hirzebruch varieties and novel K3-fibrations}, Adv. Theore. Math. Phys., {\bf 22},  (2018), 261-303.
 
 \bibitem[Be09]{bertin}   M.A. Bertin,
\textit{Examples of Calabi-Yau 3-folds of $\mathbb{P}^7$ with $\rho=1$}.
Canad. J. Math., {\bf 61} (5), (2009), 1050--1072.

 






\bibitem[BI16]{bi}  G. Bini and D. Iacono, \emph{Diffeomorphism classes of Calabi-Yau varieties},  Rend. Sem. Mat. {\bf 73} (Issue 1-2), (2015), 9--20.

\bibitem[Bo78]{bogomolov}
F. Bogomolov, \emph{Hamiltonian K\"{a}hlerian manifolds,} Dokl.
Akad. Nauk SSSR,  \textbf{243},  (1978), 1101-1104, Soviet Math.
Dokl., \textbf{19}, (1979), 1462--1465.
 
\bibitem[CLS11]{coxlittle} D.A. Cox,  J.B. Little and H.K.  Schenck, \emph{Toric varieties},
Graduate Studies in Mathematics, \textbf{124}, American Mathematical Society, Providence, RI, 2011.

\bibitem[De01]{debarre} O. Debarre, \emph{Higher-Dimensional Algebraic Geometry}, Universitext, 
 Springer-Verlag, New York, 2001.
 
 \bibitem[Do82]{dolga} I. Dolgachev, \emph{Weighted projective varieties}, in Group actions and vector fields, Proc. Vancouver 1981, LNM \textbf{956}, 34--71, Springer-Verlag, New York, 1982.

\bibitem[Ei95]{eisenbud} D. Eisenbud, \emph{Commutative algebra. With a view toward algebraic geometry,} Graduate Texts in Mathematics, \textbf{150}, Springer-Verlag, New York, 1995.

 \bibitem[GHL14]{ghl} J. Gray, A.S. Haupt and A. Lukas,
\emph{Topological Invariants and Fibration Structure of Complete Intersection Calabi-Yau Four-Folds},    JHEP, \textbf{9}, (2014), 093. 




\bibitem[GH87a]{green} P. Green and T. H\"ubsch,
\emph{Calabi-Yau Manifolds as Complete Intersections in Products of Complex Projective Spaces}, Commun. Math. Phys., \textbf{109},
(1987), 99--108.


 \bibitem[GH87b]{green2}  
P. Green, T. H\"ubsch, \emph{Polynomial deformations and cohomology of Calabi-Yau
manifolds}, Commun. Math. Phys., {\bf 113},  (1987), 505-528.


 \bibitem[Hi95]{hst}  F. Hirzebruch, 
 \emph{Topological methods in algebraic geometry}. Classics in Mathematics. Translation from the German and appendix one by R. L. E. Schwarzenberger. Appendix two by A. Borel (Reprint of the 2nd, corr. print. of the 3rd ed.), Springer-Verlag, Berlin, 1995.

\bibitem[Ho76]{hori} E. Horikawa,
\emph{On deformations of holomorphic maps III, }   Math. Annalen,
{\bf 222}, (1976), 275-282.

 \bibitem[Hu95]{huy}  D. Huybrechts,
 \emph{The tangent bundle of a Calabi-Yau manifold—deformations and restriction to rational curves}.
Comm. Math. Phys., {\bf171} no. 1, (1995), 139--158.

  \bibitem[IF00]{fletcher} A.R. Iano-Fletcher,
 \emph{Working with weighted complete intersection}, in 
Explicit birational geometry of 3-folds, London Math. Soc., Lecture Note Ser.  \textbf{281}, Cambridge Univ. Press, Cambridge, (2000), 101--173.
 
 \bibitem[Kaw92]{kawa}
Y. Kawamata, \emph{Unobstructed deformations - a remark on a paper
of Z.~Ran,} J.~Algebraic Geom., \textbf{1}, (1992), 183--190.

 \bibitem[Ko06]{kollar}
J. Koll\'ar, \emph{Non-quasi-projective moduli spaces}, Ann. of Mathematics, \textbf{164}, (2006), 1077--1096.

\bibitem[IM10]{algebraicBTT}
D. Iacono and M. Manetti, \emph{An algebraic proof of Bogomolov-Tian-Todorov theorem,}
Deformation Spaces, \textbf{39}, (2010), 113-133;
\texttt{arXiv:0902.0732v2 [math.AG]}.

\bibitem[Ma09]{ManettiSeattle}
M. Manetti, \emph{Differential graded Lie algebras and  formal
deformation theory,} in \emph{Algebraic Geometry: Seattle 2005,}
Proc. Sympos. Pure Math., \textbf{80}, (2009), 785-810.

\bibitem[Pe19]{petracci}
A. Petracci, \emph{On deformations of toric Fano varieties}, to appear in Interactions with Lattice Polytopes; preprint:  \texttt{arXiv:1912.01538 [math.AG]}.


\bibitem[PO]{popa}M. Popa, Notes for 483-3: Kodaira dimension of algebraic varieties, \url{https://sites.math.northwestern.edu/~mpopa/483-3/notes/notes.pdf}.

\bibitem[Se06]{Sernesi} E. Sernesi,
\emph{Deformations of Algebraic Schemes.} Grundlehren der
mathematischen Wissenschaften, \textbf{334}, Springer-Verlag, New
York Berlin, 2006.  

\bibitem[Ra92]{zivran} Z. Ran,
\emph{Deformations of manifolds with torsion or negative canonical
bundle},  J. Algebraic Geom., \textbf{1}, (1992),   279--291.

\bibitem[Re87]{reid}
 M. Reid, 
\textit{The moduli space of 3-folds with K=0 may nevertheless be irreducible}, Math.
Annal., \textbf{278},  (1987), 329--334.

\bibitem[Ro20]{rossi}
 M. Rossi,  
\emph{An extension of polar duality of toric varieties and its consequences in Mirror Symmetry}; preprint:  \texttt{arXiv:2003.08700 [math.AG]}.





\bibitem[Ti87]{tian} G. Tian,
\emph{Smoothness of the universal deformation space of compact
Calabi-Yau manifolds and its Petersson-Weil metric,} Mathematical
Aspects of String Theory (San Diego, 1986),  Adv. Ser. Math. Phys.
1, World Sci. Publishing, Singapore, (1987), 629--646.

\bibitem[To89]{todorov} A.N. Todorov,
\emph{The Weil-Petersson geometry of the moduli space of $SU(n\ge
3)$ (Calabi-Yau) Manifolds I,} Commun. Math. Phys., \textbf{126},
(l989), 325--346.

 \bibitem[Wa06]{wang}  
S-S. Wang, \emph{On the connectedness of the standard web of Calabi-Yau 3-folds and small transitions}; preprint:  \texttt{arXiv:1603.03929 [math.AG]}.


 
  
\end{thebibliography}
\end{document}